\documentclass[reqno,11pt]{amsart}
\usepackage{amsmath,amsfonts,amscd,amssymb,latexsym}

\usepackage{hyperref}
\usepackage{soul}

\DeclareFontFamily{U}{BOONDOX-calo}{\skewchar\font=45 }
\DeclareFontShape{U}{BOONDOX-calo}{m}{n}{
  <-> s*[1.05] BOONDOX-r-calo}{}
\DeclareFontShape{U}{BOONDOX-calo}{b}{n}{
  <-> s*[1.05] BOONDOX-b-calo}{}
\DeclareMathAlphabet{\mathcalboondox}{U}{BOONDOX-calo}{m}{n}
\SetMathAlphabet{\mathcalboondox}{bold}{U}{BOONDOX-calo}{b}{n}
\DeclareMathAlphabet{\mathbcalboondox}{U}{BOONDOX-calo}{b}{n}

\usepackage{mathrsfs}  

\usepackage[nobysame]{amsrefs}
\usepackage{comment}
\usepackage{soul}
\usepackage{tikz}
\usepackage{marginnote}
\usepackage{anysize}
\usepackage{bbm}
\usepackage{cancel}
\usepackage[normalem]{ulem}

\makeatletter
\newcommand{\addresseshere}{%
  \enddoc@text\let\enddoc@text\relax
}
\makeatother

\numberwithin{equation}{section}

\theoremstyle{plain}
\newtheorem{theorem}[subsection]{Theorem}
\newtheorem{lemma}[subsection]{Lemma}
\newtheorem{corollary}[subsection]{Corollary}
\newtheorem{proposition}[subsection]{Proposition}

\theoremstyle{definition}
\newtheorem{definition}[subsection]{Definition}

\newtheorem{assumption}[subsection]{Assumption}

\theoremstyle{remark}
\newtheorem{remark}[subsection]{Remark}

\newcommand{\ind}{\operatorname{ind}}
\newcommand{\id}{\operatorname{Id}}

\newcommand{\CC}{\mathbb{C}}
\newcommand{\RR}{\mathbb{R}}

\newcommand{\ZZ}{\mathbb{Z}}

\newcommand{\upper}{\uppercase\expandafter}

\newcommand{\n}{\nabla}
\newcommand{\p}{\partial}

\newcommand{\End}{\operatorname{End}}

\newcommand{\oB}{\bar{B}}
\newcommand{\spf}{\operatorname{sf}}

\newcommand{\IM}{\operatorname{Im}}

\renewcommand{\H}{\mathscr{H}}\newcommand{\HH}{\mathcal{H}}
\renewcommand{\P}{\mathscr{P}}\newcommand{\Q}{\mathcal{Q}}
\newcommand{\D}{\mathscr{D}}\newcommand{\B}{\mathscr{B}}
\newcommand{\M}{\mathcal{M}}\newcommand{\N}{\mathcal{N}}
\newcommand{\MM}{\mathscr{M}}

\newcommand{\E}{\mathcal{E}}\newcommand{\F}{\mathcal{F}}
\renewcommand{\c}{\mathcalboondox{c}}
\newcommand{\f}{\mathcalboondox{f}}
\newcommand{\ch}{\operatorname{ch}}
\newcommand{\oM}{\overline{M}}

\newcommand{\Ib}{\mathcal{I}_{\rm Bulk}}
\newcommand{\Ie}{\mathcal{I}_{\rm Edge}}

\newcommand{\Herm}{\operatorname{Herm}}

\renewcommand{\b}{\bullet}
\newcommand{\even}{\operatorname{even}}
\newcommand{\odd}{\operatorname{odd}}
\newcommand{\op}{\bar{\p}}

\renewcommand{\>}{\rangle}

\begin{document}

\title{The spectral Flow of a family of Toeplitz operators}

\author{Maxim Braverman\\ ({\tiny{With an appendix by}} Koen van den Dungen)}
\address{Department of Mathematics,
Northeastern University,
Boston, MA 02115,
USA}

\email{maximbraverman@neu.edu}
\urladdr{https://web.northeastern.edu/braverman/}

\subjclass[2010]{58J30, 32T15, 19K56, 58J32, 58Z05	}
\keywords{spectral flow, Toeplitz, Atiyah-Patodi-Singer, index, pseudoconvex, topological insulators, bulk-boundary correspondence}
\thanks{Partially supported by the Simons Foundation collaboration grant \#G00005104.}

\begin{abstract}
We show that the (graded) spectral flow of a family of Toeplitz operators on a complete Riemannian manifold is equal to the index of a certain Callias-type operator. When the dimension of the manifold is even this leads to a cohomological formula for the spectral flow. As an application, we compute the spectral flow of a family of Toeplitz operators on a strongly pseudoconvex domain in $\CC^n$. This result is  similar to the Boutet de Monvel's computation of the index of a single Toeplitz operator on a strongly pseudoconvex domain. Finally, we show that the bulk-boundary correspondence in the tight-binding  model of topological insulators is a special case of our result. 

In the appendix, Koen van den Dungen reviewed the main result in the context of (unbounded) KK-theory.
\end{abstract}


\maketitle
\setcounter{tocdepth}{1}



\section{Introduction}\label{S:introduction}

In the study of topological insulators the {\em edge index} is often defined as the spectral flow of a certain family of Toeplitz operators on a circle \cite{Hatsugai93,KellendonkRichterSB02,ElbauGraf02,YuWuXie17bulk-edge,GrafPorta13,ProdanSchulz16book,Hayashi16}. The {\em  bulk-edge correspondence} establishes the equality of the edge index  and the {\em bulk index}, which can be interpreted as the index of a certain Dirac-type operator. Thus we obtain the  equality between the spectral flow of a family of Toeplitz operators and the index of a Dirac-type operator. This is similar but different from the classical ``desuspension" result of Baum-Douglass \cite{BaumDouglas82} and Booss-Wojciechowski \cite{BoossWojciechowski85} (see also \cite[\S17]{BoosWoj93book}) which establishes an equality between the spectral flow of a family of Dirac-type operators and the index of a Toeplitz operator. The goal of this note is to generalize the bulk-edge correspondence to a formula for the spectral flow of a quite general family of Toeplitz operators. 

We note that the result of \cite{BaumDouglas82,BoossWojciechowski85} were extended to a family case by Dai and Zhang \cite{DaiZhang98}. It would be interesting to obtain a similar extension of our results.

\subsection{A family of Toeplitz opertors}\label{SS:IToeplitz}
Suppose $E=E^+\oplus E^-$ is a Dirac bundle over a complete Riemannian manifold $M$ and let $D$  be the corresponding Dirac operator on the space $L^2(M,E\otimes\CC^k)$. We denote by $\HH= \HH^+\oplus\HH^-$ the kernel of $D$ and by $P:L^2(M,E\otimes\CC^k)\to \HH$ the orthogonal projection. Let $\f= \{f_t\}_{t\in S^1}$ be a periodic family of smooth functions on $M$ with values in the space of Hermitian $k\times{}k$-matrices. For $t\in S^1$, we  denote by $M_{f_t}:L^2(M,E\otimes\CC^k)\to L^2(M,E\otimes\CC^k)$ the multiplication by $f_t$. The {\em Toeplitz operator} is the composition $T_{f_t}:= P\circ{}M_{f_t}:\HH\to \HH$ ($t\in S^1$).

Under the assumption that 
\begin{enumerate}
\item 0 is an isolated point of the spectrum of $D$;
\item for each $t\in S^1$ the differential $df_t$ vanishes at infinity on $M$;
\item $f$ is {\em invertible at infinity} (cf. Definition~\ref{D:invertible});
\item $\frac{\p}{\p t}f_t$ is bounded;
\end{enumerate}
we show that $T_{f_t}$ is a family of Fredholm operators. Let $T_{f_t}^\pm$ denote the restriction of $T_{f_t}$ to $\HH^\pm$ and let $\spf(T_{f_t}^\pm)$ denote the spectral flow of $T_{f_t}$. In this note we compute  $\spf(T_{f_t}^+)-\spf(T_{f_t}^-)$. In many applications, including the Toeplitz operator on a strongly pseudoconvex domain,   $\HH^-=\{0\}$ and, hence, $\spf(T_{f}^-)=0$. Thus in those cases we compute $\spf(T_{f_t}^+)$.

\subsection{A Callias-type operator}\label{SS:ICallias}
Let $\M=S^1\times M$ and let $\E=\E^+\oplus\E^-$ be the lift of $E$ to $\M$. It is naturally an ungraded Dirac bundle and we denote the corresponding Dirac operator by $\D$. Let $\MM_\f:L^2(\M,\E\otimes\CC^k)\to L^2(\M,\E\otimes\CC^k)$ denote the multiplication by $\f$ and consider the {\em Callias-type operator} 
\(
	\B_{c\f} :=  \D+ic\MM_\f,
\)
where $c>0$ is a large constant.
Our assumptions guarantee that this operator is Fredholm and our main result (Theorem~\ref{T:sfToeplitz=Callias}) states that 
\begin{equation}\label{E:Ispf=Callias}
	\spf(T_{f_t}^+)\ - \  \spf(T_{f_t}^-)\ = \ \ind \B_{c\f}.
\end{equation}
In the appendix Koen van den Dungen presented a KK-theoretical interpretation of this equality. 
\subsection{The even dimensional case: a cohomological formula}\label{SS:Ievendim}
Suppose now that the dimension of $M$ is even. Then the dimension of $\M$ is odd and  by the Callias-type index theorem \cite{Anghel93Callias,Bunke95}  the index of $\B_{c\f}$ is equal to the index of a certain Dirac operator on a compact hypesurface $\N\subset \M$. Applying the Atiyah-Singer index theorem we thus obtain a  cohomological formula for $\ind\B_{c\f}$ and, hence, for $\spf(T_{f_t}^+)-\spf(T_{f_t}^-)$, cf. Corollary~\ref{C:evencase}.

\subsection{A family of Toeplitz operators on a strongly pseudoconvex domain}\label{SS:Ipseudoconvex}
Consider  a strongly pseudoconvex domain $\oM\subset \CC^n$ with smooth boundary.  We denote its boundary by $N:=\p\oM$ and consider its interior $M$ as a complete Riemannian manifold endowed with the Bergman metric,  cf. \cite[\S7]{Stein72book}. Let $D = \op+\op^*$ be the Dolbeault-Dirac operator on the space $\Omega^{n,\b}(M,\CC^k)$ of $(n,\b)$-forms on $M$ with values in the trivial bundle $\CC^n$. Let $\f=\{f_t\}_{t\in S^1}$ be a  periodic family of smooth functions on $\oM$ with values in the space of Hermitian $k\times{}k$-matrices. We assume that the restriction of $\f$ to $N:=\p\oM$ is invertible. Then it follows from  \cite[\S5]{DonnellyFefferman83} that $D$ and $\f$ satisfy our assumptions (i)--(iv). Moreover, in this case the space $\HH^-=\{0\}$. Thus \eqref{E:Ispf=Callias} computes $\spf(T_{f_t})= \spf(T_{f_t}^-)$. 

Let $\N:=S^1\times{}N$ denote the boundary of $\overline{\M}:= S^1\times\oM$. Since the restriction of $\f$ to $\N$ is invertible, we obtain a direct sum decomposition  $\N\times\CC^k= \F_{\N+}\oplus\F_{\N-}$ of the trivial bundle $\N\times\CC^k$ into the positive and negative eigenspaces of $\f$. Our Theorem~\ref{T:sfpseudoconvex} states that 
\begin{equation}\label{E:Isfpseudconvex}
	\spf (T_{f_t}) \ = \ 
	\int_{\N}\,\ch(\F_{\N+}),
\end{equation}
where $\ch(\F_{\N+})$ denotes the Chern character of the bundle $\F_{\N+}$.

\subsection{A tight-binding model of topological insulators and the bulk-boundary correspondence}\label{SS:IGraf-Porta}
We considere a {\em tight binding model} on the lattice $\ZZ_{\ge0}\times\ZZ$. Note that this model covers not only crystals with square lattice, but many other materials, including the hexagon lattice of graphene, cf. Section~3 of \cite{GrafPorta13}.

In the bulk (i.e. far from the boundary) the lattice looks like $\ZZ\times\ZZ$. The {\em bulk Hamiltonian} is a bounded map $H:l^2(\ZZ\times\ZZ,\CC^k)\to l^2(\ZZ\times\ZZ,\CC^k)$  which is periodic with period one in both directions of the lattice. By performing the Fourier transform  of $H$ (the {\em Bloch decomposition}) we obtain a family of self-adjoint $k\times k$-matrices $H(s,t)$ ($(s,t)\in S^1\times S^1$). 

{\em We assume that the bulk Hamiltonian has a spectral gap at Fermi level $\mu\in \RR$}. In particular, the operator $H(s,t)-\mu$ is invertible  for all $(s,t)\in S^1\times S^1$. Thus the trivial bundle $(S^1\times{}S^1)\times\CC^k$ over the torus $S^1\times{}S^1$ decomposes into the direct sum
\(
	(S^1\times{}S^1)\times\CC^k =  \F_+\oplus \F_-
\)
of positive and negative eigenbundles of $H(s,t)-\mu$. The bundle $\F_+$ is referred to as the {\em Bloch bundle}. The {\em bulk index} is defined by
\begin{equation}\label{E:Ibulkindex}\notag
	\Ib\ := \ \int_{S^1\times S^1}\,\ch(\F_+).
\end{equation}

We now take the edge into account, i.e. restrict to the half-lattice $\ZZ_{\ge0}\times\ZZ$. The Fourier transform of the bulk Hamiltonian $H$ in the direction ``along the edge" transforms $H$ into a family of self-adjoint translationally invariant operators $H(t):l^2(\ZZ,\CC^k)\to  l^2(\ZZ,\CC^k)$.  Let $\Pi:l^2(\ZZ,\CC^k)\to l^2(\ZZ_{\ge0},\CC^k)$ denote the projection. The  {\em edge Hamiltonian} is the family of Toeplitz operators
\begin{equation}\label{E:IegdeHamiltonian}\notag
	H^\#(t)\  := \ \Pi\circ H(t)\circ \Pi:\,
	l^2(\ZZ_{\ge0},\CC^k)\ \to \ l^2(\ZZ_{\ge0},\CC^k),
	\qquad t\in S^1.
\end{equation}
The {\em edge index}\/ $\Ie$\/ of  the Hamiltonian $H$ is the spectral flow of the edge Hamiltonian
\begin{equation}\label{E:Iedgeindex}\notag
	\Ie\ := \ \spf\big(H^\#(t)\big). 
\end{equation}
Both the bulk and the edge Hamiltonians extend to operators on the unit disc $B\subset \CC$.  The disc is the simplest example of a strongly pseudoconvex domain. Applying \eqref{E:Isfpseudconvex} to this situation we obtain the following {\em bulk-boundary correspondence} equality
\begin{equation}\label{E:Ibulkboundary}
		\Ib\ = \ \Ie.
\end{equation}
Thus \eqref{E:Isfpseudconvex} is an extension of \eqref{E:Ibulkboundary} to general strongly pseudoconvex domains. For this reason we refer to the equality \eqref{E:Isfpseudconvex} as the {\em generalized bulk-boundary correspondence}.

\subsection*{Acknowledgements}
I would like to thank Jacob Shapiro for interesting discussion and bringing some references to my attention.

\section{The main result}\label{S:main}

In this section we formulate our main result -- the equality between the spectral flow of a family of Toeplitz operators on a complete Riemannian manifold $M$ and the index of a Callias-type operator on $M$. 

\subsection{The Dirac operator}\label{SS:Dirac}
Let $M$ be a complete Riemannian manifold and let $E=E^+\oplus E^-$ be a graded Dirac bundle over $M$, cf. \cite[\S{}II.5]{LawMic89}, i.e. a Hermitian vector bundle endowed with the Clifford action 
\[	
	c:T^*M\to \End(E), \qquad 
	\big(\,
	c(v):E^\pm\to E^\mp, \quad c(v)^2=-|v|^2, \quad c(v)^*=-c(v)
	\,\big),
\]
and a Hermitian connection $\n^E=\n^{E^+}\oplus\n^{E^-}$, which is compatible with the Clifford action in the sense that 
\[
	[\n^E_u,c(v)]\ = \ c(\n^{LC}_uv), \qquad\text{for all}
	 \ \ u\in TM,
\]
where $\n^{LC}$ is the Levi-Civita connection on $T^*M$. 

We extend the Clifford action to the product $E\otimes\CC^k$ and we denote by  $D$ the associated Dirac operator. In local coordinates it can be written as $D=\sum_jc(dx^j)\n^E_{\p_j}$. We view $D$ as an unbounded  self-adjoint operator $D:L^2(M,E\otimes\CC^k)\to L^2(M,E\otimes\CC^k)$.

Throughout the paper we make the following

\begin{assumption}\label{A:1}
Zero is an isolated point of the spectrum of $D$. 
\end{assumption}

Let $\HH=\HH^+\oplus \HH^-\subset L^2(M,E\otimes\CC^k)$ denote the kernel of $D$ and let $P:L^2(M,E\otimes\CC^k)\to \HH$ be the orthogonal projection. Here $\HH^\pm$ is the restriction of $\HH$ to $L^2(M,E^\pm\otimes\CC^k)$. We denote by $P^\pm:L^2(M,E^\pm\otimes\CC^k)\to \HH^\pm$ the restriction of $P$ to $L^2(M,E^\pm\otimes\CC^k)$.

\subsection{The Toeplitz operator}\label{SS:Toeplitz}
Let $BC(M;k)$ denote the Banach space of bounded continuous functions on $M$ with values in the space  $\Herm(k)$ of Hermitian complex-valued $k\times{}k$-matrices. 
We denote by $C^\infty_g(M;k)\subset BC(M;k)$ the subspace of bounded $C^\infty$-functions such that $df$ vanishes at infinity of $M$.

For  $f\in C^\infty_g(M;k)$ we denote by $M_f:L^2(M,E\otimes\CC^k)\to L^2(M,E\otimes\CC^k)$ the multiplication by $1\otimes{f}$ and by $M_f^\pm$ the restriction of $M_f$ to $L^2(M,E^\pm\otimes\CC^k)$.

\begin{definition}\label{D:Toeplitz}
The operator 
\begin{equation}\label{E:Toeplitz}
	T_f\ := \ P\circ M_f:\, \HH\ \to \HH
\end{equation}
is called the {\em Toeplitz operator defined by $f$}.  We denote by $T_f^\pm$ the restriction of $T_f$ to $\HH^\pm$. 
\end{definition}

\begin{definition}\label{D:invertible}
We say that a matrix-valued function $f\in C^\infty_g(M;k)$ is {\em invertible at infinity} if there exists a compact set $K\subset M$ and $C_1>0$  such that $f(x)$ is an invertible matrix for all $x\not\in K$ and $\big\|f(x)^{-1}\big\|<C_1$ for all $x\not\in K$.
\end{definition}

The following result is proven in \cite[Lemma~2.6]{Bunke00}
\begin{proposition}\label{P:Fredholm}
If $D$ satisfies Assumption~\ref{A:1} and  $f\in C^\infty_g(M;k)$  is invertible at infinity then the Toeplitz operator $T_f$ is Fredholm. 
\end{proposition}

\subsection{A family of self-adjoint Toeplitz operators}\label{SS:family}
Let now $S^1= \{e^{it}:t\in [0,2\pi]\}$ be the unit circle. Consider a smooth  family $f:S^1\to C^\infty_g(M;k)$, $t\mapsto f_t$, of matrix valued functions. Assume that $f_t$ is invertible at infinity for all $t\in [0,2\pi]$.  Then $T_{f_t}$ ($t\in S^1$) is a periodic family of self-adjoint Fredholm operators. Our goal is to compute the spectral flow of this family.  We make the following 

\begin{assumption}\label{A:2}
There exists a constant $C_2>0$ such that $\big\|\frac{\p}{\p t}f_t(x)\big\|< C_2$ for all $t\in S^1,\ x\in M$. 
\end{assumption}	

If $f_t$ is invertible at infinity for all $t\in S^1$ and satisfies Assumption~\ref{A:2}, then there exists a compact set $K\subset M$ and  a large enough constant $\alpha>0$ such that   
\begin{equation}\label{E:ddtf<alpha}
	\frac{\p}{\p t}f_t(x)\ < \ 
	\frac{\alpha}2\,f_t(x)^2,
	\qquad\text{for all} \ \ x\not\in K.
\end{equation}
Here the inequality $A<B$ between two self-adjoint matrices means that for any vector $v\not=0$, we have $\<Av,v\><\<Bv,v\>$.

\subsection{A Callias-type operator on $S^1\times M$}\label{SS:CalliasSxM}Set $\M= S^1\times M$. We write points of $\M$ as $(t,x)$, $t\in S^1, \ x\in M$. Denote by $\pi_1:S^1\times M\to S^1$ and $\pi_2:S^1\times M\to M$ the natural projections. By a slight abuse of notation we denote the pull-backs $\pi^*_1dt,\,\pi_2^*dx\in T^*\M$ by  $dt$ and $dx$ respectively. Set  
\[
	\E := \ \pi_2^*E.
\] 
Then $\E$ is naturally an ungraded Dirac bundle with Clifford action $\c:T^*\M\to \End(\E)$ such that $\c(dx)= c(dx)$ and $\c(dt)$ is given with respect to the decomposition $\E=\pi_2^*E^+\oplus\pi_2^*E^-$ by the matrix
\[
	\c(dt)\ = \ \begin{pmatrix}
	i\cdot\id&0\\0&-i\cdot\id
	\end{pmatrix}.
\] 
Let $\D$ be the corresponding Dirac operator. With respect to the decomposition 
\begin{equation}\label{E:L2=L2timesL2}
		L^2(\M,\E\otimes\CC^k)\ = \ 
		L^2(S^1)\otimes L^2(M,E\otimes\CC^k).
\end{equation}
it takes the form 
\begin{equation}\label{E:tilD}
	\D\ = \ \c(dt)\,\frac{\p}{\p t}\otimes1\ + \ 1\otimes D.
\end{equation}
We remark that, as opposed to $D$, the operator $\D$ is not graded. 

Let now $f_t\in C^\infty_g(M;k)$ ($t\in S^1$) be a smooth  periodic family of invertible at infinity matrix valued functions satisfying Assumption~\ref{A:1}.  We consider the family $f_t$ as a smooth function on $\M$ and denote it by $\f$. Let $\MM_\f:L^2(\M,\E\otimes\CC^k)\to L^2(\M,\E\otimes\CC^k)$ denote the multiplication by $\f$. Then   the commutator 
\[
	[\D,\MM_\f]\ := \ \D\circ\MM_\f-\MM_\f\circ\D
\] 
is a zero-order differential operator, i.e.,  a bundle map $\E\otimes\CC^k\to \E\otimes\CC^k$.


From \eqref{E:ddtf<alpha} and our assumption that $df_t$ vanishes at infinity, we now conclude that there exist  constants $c,d>0$ and  a  compact set $\mathcal{K}\subset \M$, called {\em an essential support} of $\B_{c\f}$, such that 
\begin{equation}\label{E:Calliascond}
	\big[\D,c\,\MM_\f\big](t,x) \ <  \ c^2{\MM_\f(t,x)}^2-d,
	\qquad \text{for all}\ \ (t,x)\not\in \mathcal{K}.
\end{equation}
It follows that 
\begin{equation}\label{E:Bft}
	\B_{c\f}\ := \ \D\ + \ i\,c\,\MM_\f
\end{equation}
is a {\em Callias-type operator} in the sense of \cite{Anghel93Callias,Bunke95} (see also \cite[\S2.5]{BrCecchini17}). In particular, it is Fredholm.

Our main result is the following

\begin{theorem}\label{T:sfToeplitz=Callias}
Let $E=E^+\oplus{}E^-$ be a Dirac bundle over a complete Riemannian manifold $M$ and let $f_t\in C^\infty_g(M;k)$ ($t\in S^1$) be a smooth  periodic family of invertible at infinity matrix-valued functions. Suppose that Assumptions~\ref{A:1} and \ref{A:2} are satisfied. Then 
\begin{equation}\label{E:sfToeplitz=Callias}
	\spf(T_{f_t}^+)\ - \  \spf(T_{f_t}^-)\ = \ \ind{\B_{c\f}}.
\end{equation}
\end{theorem}

\begin{remark}\label{R:nonvanishing}
Note that, as opposed to $df$, the differential $d\f$ does not vanish at infinity. Because of this $\B_{c\f}$ does not satisfy the conditions of Corollary~2.7 of \cite{Bunke00} and its index does not vanish in general. 
\end{remark}

\begin{remark}\label{R:vanishing of sfT-}
In our main applications $H^-=\{0\}$. Hence, $\spf(T_{f_t}^-)=0$ and \eqref{E:sfToeplitz=Callias} computes $\spf(T_{f_t}^+)$.
\end{remark}

The proof of Theorem~\ref{T:sfToeplitz=Callias} is given in Section~\ref{S:prsfToeplitz=Callias}.

\subsection{The even dimensional case: a cohomological formula}\label{SS:evendim}
Suppose now that the dimension of $M$ is even. Then the dimension of $\M$ is odd and  by the Callias-type index theorem \cite{Anghel93Callias,Bunke95}  the index of $\B_{c\f}$ is equal to the index of a certain Dirac operator on a compact hypesurface $\N\subset \M$. Applying the Atiyah-Singer index theorem we thus obtain a  cohomological formula for\/ $\ind\B_{c\f}$. We now provide the details of this computation.

Let $N\subset M$ be a hypersurface such that there is an essential support $\mathcal{K}\subset \M$ of $\B_{c\f}$ whose boundary $\p{}\mathcal{K}= \N:=S^1\times{}N$. In particular, the restriction of $\f$ to $\N$ is invertible and satisfies \eqref{E:Calliascond}. Then there are vector bundles $\F_{\N\pm}$ over $\N$ such that 
\[
	\M\times\CC^k \ = \ \F_{\N+}\oplus \F_{\N-},
\]
and the restriction of $\f$ to $\F_{\N+}$ (respectively, $\F_{N-}$) is positive definite (respectively, negative definite).

\begin{corollary}\label{C:evencase}
Under the conditions of Theorem~\ref{T:sfToeplitz=Callias} assume that $\dim{}M$ is even. Let $j:N\hookrightarrow M$ be a hypersurface such that there is an essential support $\mathcal{K}\subset \M$ of $\B_{c\f}$ whose boundary $\p{}\mathcal{K}= \N:=S^1\times{}N$. Then 
\begin{equation}\label{E:sfToeplitz=cohom}
	\spf(T_{f_t}^+)\ - \  \spf(T_{f_t}^-)\ = \ 
	\int_N\,\Big[\,
	j^*\hat{A}(M)\,j^*\ch(E/S)\,\pi_{2*}\ch(\F_{\N+})\,\Big].
\end{equation}
Here $\hat{A}(M)$ is the differential form representing $\hat{A}$-class of $M$, $\ch(E/N)$ is the differential form representing the {\em graded relative Chern character} of $E$, cf. \cite[p.~146]{BeGeVe}, and 
\[
	\pi_{2*}\ch(\F_{\N+})\ = \ \int_{S^1}\,\ch(\F_{\N+})
	\ \in \ \Omega^\b(N)
\]
is the push-forward of $\ch(\F_{\N+})$ under the map $\pi_2:\N\to N$.   
\end{corollary}

\begin{proof}
Let $E_N$ denote the restriction of the bundle $E$ to $N$. Then  $\E_\N:=\pi_2^*E_N$ is the restriction of $\E$ to $\N$. It is naturally a  Dirac bundle over $\N$. We denote by $\D_\N$ the induced Dirac-type operator on $\E_\N\otimes \F_{\N+}$.

Let $v$ denote the unit normal vector to $\N$ pointing towards $\mathcal{K}$. Then  $i\c(v):\E_\N\to \E_\N$ is an involution. We denote by $\E_\N^{\pm1}$ the eigenspace of $i\c(v)$ with eigenvalue $\pm1$. This defines a grading $\E_\N= \E_\N^{+1}\oplus\E_\N^{-1}$  on the Dirac bundle $\E_N$. 
The operator $\D_\N$ is odd with respect to the induced grading on $\E_\N\otimes \F_{\N+}$.  (This grading is different from the one induced by the grading on $E$. Note that the operator $\D_\N$  is not an odd operator with respect to the grading induced by the grading on $E$).

By the Callias-type index theorem \cite{Anghel93Callias,Bunke95} (see also \cite[\S2.6]{BrCecchini17} where more details are provided)
\begin{equation}\label{E:indB=indD}
		\ind \B_{c\f}\ = \ \ind \D_{\N}.
\end{equation}

Since $\D_\N$ is an operator on compact manifold $\N$ its index is given by the Atiyah-Singer index theorem. Combining it with \eqref{E:indB=indD} we obtain
\begin{equation}\label{E:indB=integral}
	\ind \B_{c\f}\ = \ 
	\int_\N\,\hat{A}(\N)\,\ch(\F_{\N+})\,\ch(\E_\N/\mathcal{S}_\N), 
\end{equation}
where $\hat{A}(\N)$ is the $\hat{A}$-genus of $\N$, $\ch(\F_{\N+})$ is the Chern character of $\F_{\N+}$, and $\ch(\E_\N/\mathcal{S}_\N)$ is the {\em relative Chern character} of the graded bundle $\E_\N$, cf. \cite[p.~146]{BeGeVe}.

Since all the structures are trivial along $S^1$, 
\begin{equation}\label{E:pistar}
		\hat{A}(\N)\ = \ \pi_2^*\hat{A}(N), \qquad
	\ch(\E_\N/\mathcal{S})\ = \ \pi_2^*\ch(E_N/S_N)
\end{equation}
where $\hat{A}(N)$ is the $\hat{A}$-genus of $N$ and $\ch(E_N/S_N)$ is the relative Chern character of the graded bundle $E_N= E_N^{+1}\oplus{}E_\N^{-1}$. 

Since $\hat{A}(N)$ is  a characteristic class it behaves naturally with respect to the pull-backs, i.e., $\hat{A}(N)=j^*\hat{A}(M)$. Combining with \eqref{E:pistar} we obtain
\begin{equation}\label{E:restrictinAhat}
	\hat{A}(\N)\ = \ \pi_2^*j^*\hat{A}(M).
\end{equation}

As for $\E_\N$, the grading  $E_N= E_N^{+1}\oplus{}E_\N^{-1}$ is different from the grading $E_N=E_N^+\oplus{}E_N^-$ inherited from the grading on $E$. However,  since $ic(v)$ is odd with respect to the grading $E_N=E_N^+\oplus{}E_N^-$ and $c(v)^2=-1$, we have
\[
	E_N^{\pm1}\ = \  \big\{\,
		e\pm ic(v)e:\,e\in E_N^+\,\big\}.
\]
It follows that  $E_N^+\oplus{}E_N^-$ and $E_N^{+1}\oplus{}E_N^{-1}$ are isomorphic as graded Dirac bundles. Hence, we can compute $\ch(E_N/S_N)$ using the grading $E_N=E_N^+\oplus{}E_N^-$. Even though the relative Chern character is not quite a characteristic class (it depends not only on the connection but also on the Clifford action and the Riemannian metric) it is well known that it behaves naturally under restrictions to a submanifold
\[
	\ch(E_N/S_N)\ =\  j^*\ch(E/S),
\]
see, for example, \cite[Lemma~7.1]{BrMaschler17}. Thus, using  \eqref{E:pistar} we now obtain
\begin{equation}\label{E:restrtictionChern}
	\ch(\E_\N/\mathcal{S})\ = \ \pi_2^*j^*\ch(E/S).
\end{equation}

The equality \eqref{E:sfToeplitz=cohom} follows now from \eqref{E:sfToeplitz=Callias}, \eqref{E:indB=integral}, \eqref{E:restrictinAhat}, and \eqref{E:restrtictionChern}.
\end{proof}

\section{A family of Toeplitz operators on a strongly pseudoconvex domain}\label{S:pseudoconvex}

In this section we apply Theorem~\ref{T:sfToeplitz=Callias} to the case when $M$ is a strongly pseudoconvex domain in $\CC^k$. Our computation of the spectral flow in this case is similar to the computation of index in \cite{BoutetdeMonvel78} and \cite{GuentnerHigson96}.

\subsection{A Dirac operator on a strongly pseudoconvex domain}\label{SS:Diracpseudoconvex}

Let $M$ be a strongly pseudoconvex domain in $\CC^n$. We denote its boundary by $N:=\p{}\oM$. Let $g^M$ be the Bergman metric on $M$, cf. \cite[\S7]{Stein72book}.  Then $(M,g^M)$ is a complete K\"ahler manifold. We define $E= \Lambda^{n,\bullet}(T^*M)$  and set $E^+= \Lambda^{n,\even}(T^*M)$, $E^-= \Lambda^{n,\odd}(T^*M)$. Then $E$ is naturally a Dirac bundle over $M$ whose space of smooth section coincides with the Dolbeault complex $\Omega^{n,\b}(M)$ of $M$ with coefficients in the canonical bundle $K=\Lambda^{n,0}(T^*M)$. Moreover, the corresponding Dirac operator is given by 
\[
	D\ = \ \op \ + \ \op^*, 
\]
where $\op$ is the Dolbeault differential and $\op^*$ its adjoint with respect to the $L^2$-metric induced by the Bergman metric on $M$. 

By \cite[\S5]{DonnellyFefferman83} zero is an isolated point of the spectrum of $D$ and $\HH:=\ker D$ is a subset of the space $(n,0)$
\begin{equation}\label{E:kernelBergman}
	\HH\ := \ \ker D\ \subset \ \Omega^{n,0}(M).
\end{equation}
In particular, Assumption~\ref{A:1} is satisfied and $\HH^-=\{0\}$.

\subsection{A family of Toeplitz operators}\label{SS:ToeplitzDomain}
Let  $C^\infty(\oM;k)$ denote the space of smooth functions on $\oM$ with values in the space of complex-valued self-adjoint $k\times{}k$ matrices. Each $f\in C^\infty(\oM;k)$ induces a function on $(M,g^M)$ which we also denote by $f$. One readily sees that $f\in C^\infty_g(M,k)$, cf. \cite[\S1, Lemma~2]{GuentnerHigson96}. Then $f$ is invertible at infinity iff $f|_{\p M}$ is invertible. 

We now consider the product $\overline{\M}:= S^1\times\oM$. This is a compact  manifold with boundary $\N:=S^1\times{N}$. We endow the interior $\M:=S^1\times{}M$  of $\overline{\M}$ with the product of the standard metric on $S^1$ and the Bergman metric on $M$. Let
$C^\infty(\overline{\M};k)$ denote the space of smooth functions on $\overline{\M}$ with values in the space of self-adjoint $k\times{}k$-matrices. For $\f\in C^\infty(\overline{\M})$, set $f_t(x):= \f(t,x)$ ($t\in S^1,\ x\in \oM$). Then the restriction of $f_t$ to $\M$ (which is also denoted by $f_t$) is a smooth family of functions in $C^\infty_g(M,k)$. Let $T_{f_t}$ be the corresponding family of Toeplitz operators on $M$. By \eqref{E:kernelBergman}, the space $\HH^-=\{0\}$. Hence, $T_{f_t}= T_{f_t}^+$, while $T_{f_t}^-=0$. 

Assume now that the restriction of $\f$ to $\N=S^1\times{\p\oM}$ is invertible. Then $f_t$ are invertible at infinity  for all $t\in S^1$.  Assumption~\ref{A:2} is automatically satisfied in this case, since $\frac{\p}{\p t} f_t$ is a continuous function on $\overline{\M}$.

As we will see in the next section the following theorem generalizes the bulk-edge correspondence in the theory of topological insulators. More precisely, in the case when $M$ is the unit disc in $\CC$, the left hand side of \eqref{E:sfpseudconvex} is equal to the edge index, while  itsthe write hand side is the bulk index.

\begin{theorem}[Generalized bulk-edge correspondence]\label{T:sfpseudoconvex}
Let $M\subset \CC^k$ be a strongly pseudoconvex domain with smooth boundary $N:=\p M$.  Set\/ $\overline{\M}:=S^1\times\oM$ and let\/ $\f\in  C^\infty_g(\overline{\M};k)$ be a smooth function with values in the space of self-adjoint $k\times{}k$-matrices. Assume that $f_t(x):= \f(t,x)$ is invertible for all $t\in S^1, \ x\in \p{}M$. Then 
\begin{equation}\label{E:sfpseudconvex}
	\spf (T_{f_t}) \ = \ 
	\int_{\N}\,\ch(\F_{\N+}),
\end{equation}
where $\N:=S^1\times{N}$ and  $\F_{\N+}\subset \N\times\CC^k$ is a subbundle spanned by eigenvectors of $\f|_\N$ with positive eigenvalues.
\end{theorem}
\begin{proof}
By \eqref{E:kernelBergman}, the space $\HH^-=\{0\}$. Hence, $T_{f_t}= T_{f_t}^+$, while $T_{f_t}^-=0$. In particular, 
\begin{equation}\label{E:sfTf-}
	\spf(T_{f_t}^-)\ = \ 0.
\end{equation}

Since $\M$ is a domain in a flat space $\CC^k$, both $\hat{A}(\M)=1$ and $\ch(E/S)=1$. Hence, by \eqref{E:sfToeplitz=cohom} and \eqref{E:sfTf-} we obtain 
\[
	\spf (T_{f_t}) \ = \ \int_N\,\pi_{2*}\ch(\F_{\N+}) \ = \
	\int_{\N}\,\ch(\F_{\N+}).
\]
\end{proof}

\section{A tight-binding model of topological insulators and the bulk-edge correspondence}\label{S:Graf-Porta}

In this section we briefly review a standard tight-binding model for two-dimensional topological insulators, following the description in \cite{Hayashi16} (see also \cite{GrafPorta13}) and show that the bulk-edge correspondence for this model follows immediately from our Theorem~\ref{T:sfpseudoconvex}.

\subsection{The bulk Hamiltonian}\label{SS:bulkGrafPorta}
We consider considered a {\em tight binding model} on the lattice $\ZZ_{\ge0}\times\ZZ$. Basically, this means that the 
electrons can only stay on the lattice sites and the kinetic energy is included by allowing electrons to hop from one site to a neighboring one. Surprisingly, this model covers not only crystals with square lattice, but many other materials, including the hexagon lattice of graphene, cf. Section~3 of \cite{GrafPorta13}. 

The mathematical formulation of the model is as follows (we present the version suggested in \cite{Hayashi16}): The  ``bulk" state space is the space $l^2(\ZZ\times\ZZ,\CC^k)$ of square integrable sequences 
\[
	\phi\ = \ \big\{\phi_{ij}\big\}_{(i,j)\in\ZZ\times\ZZ}\,,
	\qquad \phi_{ij}\in \CC^k.
\]
The {\em bulk Hamiltonian} $H:l^2(\ZZ\times\ZZ,\CC^k)\to l^2(\ZZ\times\ZZ,\CC^k)$ is periodic with period one in both directions of the lattice. By performing a Fourier transform  of $H$ (the {\em Bloch decomposition} in  physics terminology) we obtain a family of self-adjoint $k\times k$-matrices $H(s,t)$ ($(s,t)\in S^1\times S^1$). {\em We assume that $H(s,t)$ depend smoothly on $s$ and $t$}.

{\em We assume that the bulk Hamiltonian has a spectral gap at Fermi level $\mu\in \RR$}, i.e. there exists $\epsilon>0$ such that the spectrum of $H(s,t)$ does not intersect the interval $(\mu-\epsilon,\mu+\epsilon)$ for all $(s,t)\in S^1\times S^1$. In particular, the operator $H(s,t)-\mu$ is invertible  for all $(s,t)\in S^1\times S^1$. Thus the trivial bundle $(S^1\times{}S^1)\times\CC^k$ over the torus $S^1\times{}S^1$ decomposes into the direct sum of subbundles 
\[
	(S^1\times{}S^1)\times\CC^k\ = \ \F_+\oplus \F_-
\]
such that the restriction of $H(s,t)-\mu$ to $\F_+$  is positive definite and the restriction of  $H(s,t)-\mu$ to $\F_-$ is negative define. The bundle $\F_+$ is referred to as the {\em Bloch bundle}.

\begin{definition}\label{D:bulkindex}
The {\em bulk index} of the Hamiltonian $H$ is 
\begin{equation}\label{E:bulkindex}
	\Ib\ := \ \int_{S^1\times S^1}\,\ch(\F_+).
\end{equation}
\end{definition}

\subsection{The edge Hamiltonian}\label{SS:edgeGrafPorta}
We now take the boundary into account. The ``edge" state space is the space  $l^2(\ZZ_{\ge0}\times\ZZ,\CC^k)$ of square integrable sequences of vectors in $\CC^k$ on the half-lattice $\ZZ_{\ge0}\times\ZZ$. 

The Fourier transform of the bulk Hamiltonian $H:l^2(\ZZ\times\ZZ,\CC^k)\to l^2(\ZZ\times\ZZ,\CC^k)$ in the direction ``along the edge" transforms $H$ into a family of self-adjoint translationally invariant operators
\begin{equation}\label{E:H(t)}
	H(t):\,l^2(\ZZ,\CC^k)\ \to \ l^2(\ZZ,\CC^k).
\end{equation}
Then $H(t)$ depends smoothly on $t\in S^1$. Let $\Pi:l^2(\ZZ,\CC^k)\to l^2(\ZZ_{\ge0},\CC^k)$ denote the projection.

\begin{definition}\label{D:edgeHamiltonian}
The {\em edge Hamiltonian} is the family of Toeplitz operators
\begin{equation}\label{E:egdeHamiltonian}
	H^\#(t)\  := \ \Pi\circ H(t)\circ \Pi:\,
	l^2(\ZZ_{\ge0},\CC^k)\ \to \ l^2(\ZZ_{\ge0},\CC^k),
	\qquad t\in S^1.
\end{equation}
\end{definition}

\begin{definition}\label{D:edgeindex}
The {\em edge index} $\Ie$ of  the Hamiltonian $H$ is the spectral flow of the edge Hamiltonian
\begin{equation}\label{E:edgeindex}
	\Ie\ := \ \spf\big(H^\#(t)\big). 
\end{equation}
\end{definition}
\begin{theorem}[Bulk-edge correspondence]\label{T:bulkedge}
\(\displaystyle
	\Ib\ = \ \Ie.
\)
\end{theorem}
\begin{proof}
Consider the unit disc $\oB:= \{z\in \CC:\, |z|\le 1\}$. This is a strongly pseudoconvex domain in $\CC$. We view the bulk Hamiltonian $H(s,t)$ as a function on $\p\oB\times{}S^1$ with values in the set $\Herm(k)$ of invertible Hermitian $k\times{}k$-matrices. 

We endow $B$ with the Bergman metric and consider the Dolbeault-Dirac  operator $D=\op+\op^*$ on the space $L^2\Omega^{1,\b}(B,\CC^k)=  L^2\Omega^{1,0}(B,\CC^k)\oplus L^2\Omega^{1,1}(B,\CC^k)$ of square-integrable $(1,\b)$-forms on $B$. Let $\HH:= \ker D$. Then $\HH\subset L^2\Omega^{1,0}(B,\CC^k)$. Let $P:L^2\Omega^{1,\b}(B,\CC^k)\to \HH$ denote the orthogonal projection. As usual, for $f\in C^\infty_b(\oB,k)$ we denote by $T_f:=P\circ M_f:\HH\to \HH$ the Toeplitz operator. 

To each $u= \{u_j\}_{j\ge0}\in l^2(\ZZ_{\ge0},\CC^k)$ we associate a 1-form
\[
	\phi(f)\ := \ \sum_{j\ge0}\,u_j\,z^j\,dz\ \in \ \HH.
\]
If $A:l^2(\ZZ,\CC^k)\to l^2(\ZZ,\CC^k)$ is a translationally invariant operator, then $\phi\circ{A}\circ\phi^{-1}$ is the multiplication operator  by the {\em Fourier transform $A(s)$ of $A$}. Let $a:\oB\to \Herm(k)$ be a continuous extension of $A(s)$ to $\oB$ (i.e, we assume that $a|_{\p\oB}(s)=A(s)$) and let 
\[
 	T_a\ := \ P\circ M_{a(s)}:\,\HH\ \to\ \HH
\]
be the corresponding Toeplitz operator. 

We also define a different Toeplitz operator  
\[
	A^\#\ :=\ \Pi\circ H(t)\circ\Pi:l^2(\ZZ_{\ge0},\CC^k)\ \to \ l^2(\ZZ_{\ge0},\CC^k)
\] 
associated with $A$ (notice that \eqref{E:egdeHamiltonian} is a special case of this construction). Those two Toeplitz operators are closely related. In particular, it is proven in \cite{Coburn73} that the difference \[
	T_a\ - \ \phi\circ A^\#\circ \phi^{-1}
\] 
is compact. We now apply this result to the family of operators $H(t):l^2(\ZZ_{\ge0},\CC^k)\ \to \ l^2(\ZZ_{\ge0},\CC^k)$. Let $h:\oB\times{}S^1\to \Herm(k)$ be a continuous extension of $H(s,t)$ and set $h_t(s):= h(s,t)$. Then 
\[
	T_{h_t}\ - \ \phi\circ H^\#(t)\circ \phi^{-1}:\,\HH\ \to \HH, 
	\qquad t\in S^1,
\]
is a continuous family of compact operators. It follows, \cite[Proposition~1.12]{BoossWojciechowski85} (see also \cite[Proposition~17.6]{BoosWoj93book}), that
\[
	\spf(T_{h_t})\ = \ \spf\big(\, \phi\circ H^\#(t)\circ \phi^{-1}\,\big)
	\ = \ \spf\big(H^\#(t)\big).
\]
The theorem follows now from definitions of bulk and edge indexes, \eqref{E:bulkindex}, \eqref{E:edgeindex}, and Theorem~\ref{T:sfpseudoconvex}.
\end{proof}

\section{Proof of Theorem~\ref{T:sfToeplitz=Callias}}\label{S:prsfToeplitz=Callias}

The proof of Theorem~\ref{T:sfToeplitz=Callias} consists of two steps. First (Lemma~\ref{L:sf=ind}) we apply a result of Atiyah, Patodi and Singer \cite[Th.~7.4]{APS3} (see also \cite{RobbinSalomon95}) to conclude that the spectral flow $\spf(T_{f_t})$ is equal to the index of a certain operator on $\M$. Then, using the argument similar to \cite{Bunke00} we show that the latter index is equal to the index of $\B_{c\f}$.

\subsection{Functions with value in $\HH$}\label{SS:L2H}
We view the space $L^2(S^1,\HH)$ of square integrable functions with values in $\HH$ as a subspace of $L^2(\M,\E\otimes\CC^k)$. Then the family  $T_{f_t}$ naturally induces an operator $L^2(S^1,\HH)\to L^2(S^1,\HH)$ which we still denote by $T_{f_t}$. Let $T_{f_t}^\pm$ denote the restriction of $T_{f_t}$ to $L^2(S^1,\HH^\pm)$. 

\subsection{The spectral flow as an index}\label{SS:sf=index}
Atiyah, Patodi and Singer, \cite[Th.~7.4]{APS3}, proved that the spectral flow of a periodic family of elliptic differential operators $A(t)$ ($t\in S^1)$ is equal to the index of the operator $\p/\p t-A$ on $S^1$. Robbin and Salomon \cite{RobbinSalomon95} extended this equality to a much more general family of operators. Applying this result to our situation we immediately obtain

\begin{lemma}\label{L:sf=ind}
Under the assumptions of Theorem~\ref{T:sfToeplitz=Callias} we have
\begin{equation}\label{E:sf=ind}
	\spf(T_{f_t}^\pm)\ = \  
	\ind\,\big(\frac{\p}{\p t}-T_{f_t}^\pm\big)\Big|_{L^2(S^1,\HH^\pm)}.
\end{equation}
\end{lemma}

Notice that, since $\spf(T_{f_t}^\pm)=-\spf(-T_{f_t}^\pm)$, equality \eqref{E:sf=ind} is equivalent to 
\begin{equation}\label{E:sf=-ind}
	\spf(T_{f_t}^\pm)\ = \  
	-\,\ind\,\big(\frac{\p}{\p t}+T_{f_t}^\pm\big)\Big|_{L^2(S^1,\HH^\pm)}.
\end{equation}

\subsection{Harmonic sections on $\M$}\label{SS:Htorus}
Let $\H\subset L^2(\M,\E\otimes\CC^k)$ denote the kernel of $\D$ and let $\P:L^2(\M,\E\otimes\CC^k)\to \H$ be the orthogonal projection. 

We denote by $P_0:L^2(S^1)\to L^2(S^1)$ the orthogonal projection onto the subspace of constant functions. Then  with respect to decomposition \eqref{E:L2=L2timesL2} we have $\P=P_0\otimes P$. 

To simplify the notation in the computations below we write $P_0$ for $P_0\otimes1$, $Q_0$ for $(1-P_0)\otimes1$, and $P$ for $1\otimes P$. Then   the space $L^2(S^1,\HH)$ coincides with the image of the projection 
\[
	P:\, L^2(\M,\E\otimes\CC^k)\ \to\  L^2(\M,\E\otimes\CC^k).
\]
It follows that the projections $\P$ and $\Q:=1-\P$ preserve the space $L^2(S^1,\HH)$ and their restrictions this space are given by 
\begin{equation}\label{E:projectionsH}
	\P|_{L^2(S^1,\HH)}\ = \ P_0, \qquad 
	\Q|_{L^2(S^1,\HH)}\ = \ Q_0.
\end{equation}

\begin{lemma}\label{L:commutatorcompact}
The operator 
\begin{equation}\label{E:commutatorcompact}
 	[P,\M_\f]\ := \ P\circ \MM_\f\ - \ \MM_\f\circ P:\,
 	L^2(\M,\E\otimes\CC^k)\ \to \  L^2(\M,\E\otimes\CC^k)
 \end{equation} 
is compact.
\end{lemma}

\begin{proof}
The proof of Lemma~2.4 of \cite{Bunke00} extends to our situation almost without any changes. We present it here for completeness. 

By Assumption~\ref{A:1} there exists a small ball $B\subset \CC$ about 0 which does not contain non-zero points of the spectrum of $1\otimes{}D$. To simplify the notation we write $D$ for $1\otimes{}D$. 

For $\lambda$ not in the spectrum of $D$, 
let $R_D(\lambda):=(\lambda-D)^{-1}$ denote the resolvent.
By functional calculus we have
\[
	P\ = \ \frac{1}{2\pi i}\, \int_{\partial B}\, R_D(\lambda)\, d\lambda.
\]
Let $d_x\f:= df_t(x)$ be the differential of $\f$ along $M$ (so that $d\f= d_x\f+\frac{\p \f}{\p t}dt$). Then we have
\[
	[R_D(\lambda),\MM_\f] \ = \ R_D(\lambda)\,[D,\MM_\f]\,R_D(\lambda)
	\ = \ 
	R_D(\lambda)\,\c(d_x\f)\, R_D(\lambda).
\]
From Rellich's Lemma and the fact that $d_x\f$ vanishes at infinity we conclude that $c(d_x\f) R_D(\lambda)$ is compact.  Hence $[R_D(\lambda),\MM_\f]$ is also compact. It follows that 
\[
	[P,\MM_\f]\ =\ \frac{1}{2\pi i}\, 
	  \int_{\partial B}\, [R_D(\lambda),\MM_\f]\: d\lambda
\]
is compact.
\end{proof}

\begin{corollary}\label{C:PMQcompact}
The operators $P\circ \MM_\f\circ Q$ and\/ $Q\circ \MM_\f\circ P$ are compact.
\end{corollary}
\begin{proof}
The operator
\[
	P\circ \MM_\f\circ Q\ - \ Q\circ  \MM_\f\circ P \ = \ 
	P\circ \MM_\f \ - \ \MM_\f\circ P
\]
is compact by Lemma~\ref{L:commutatorcompact}. Since the range of $P$ is orthogonal to the range of $Q$, it follows that the operators $P\circ 
\MM_\f\circ\Q$ and $\Q\circ \MM_\f\circ\P$ are compact.
\end{proof}

\begin{lemma}\label{L:indddt=indBf}
Under the assumptions of Theorem~\ref{T:sfToeplitz=Callias} we have
\begin{equation}\label{E:indddt=indTf}
	 \ind \B_{c\f}\ = \ 
	 \ind\,\big(\frac{\p}{\p t}-T_{f_t}^+\big)\Big|_{L^2(S^1,\HH^+)}
	 \ + \ 
	 \ind\,\big(\frac{\p}{\p t}+T_{f_t}^-\big)\Big|_{L^2(S^1,\HH^-)}.
\end{equation}
\end{lemma}
\begin{proof}
Set 
\[
	A\ := \ \c(dt)\,\frac{\p}{\p t}\ + \ i\,c\,\MM_\f.
\]
Then $\B_{c\f}=1\otimes D+A$. Consider a one parameter family of operators
\[
	\B_{c\f,u}\ := \ 1\otimes D\ + \ u\,A,
	\qquad 0\le u\le 1.
\]
It follows from \ref{A:2} and the vanishing of $d_x\f$ at infinity that for all $u>0$ the operator $\B_{c\f,u}$ satisfies the Callias-condition \eqref{E:Calliascond} and, hence, is Fredholm. 

Since $D$ and $\frac{\p}{\p t}$ commute with $P$ and $Q$
\[
	 P\circ \B_{c\f,u}\circ Q\ = \ u\,c\, P\circ \MM_\f\circ Q, 
	 \qquad 
	 Q\circ  \B_{c\f,u}\circ P \ = \  u\,c\,Q\circ  \MM_\f\circ P.
\]
Thus these operators are compact by Corollary~\ref{C:PMQcompact}. It follows that 
\begin{equation}\label{E:indB=diag}
	\ind \B_{c\f,u}\ = \ 
	\ind P\circ \B_{c\f,u}\circ P|_{\IM P}\ + \ 
	\ind Q\circ \B_{c\f,u}\circ Q|_{\IM Q}.
\end{equation}

The operator $Q\circ \B_{c\f,0}\circ Q|_{\IM Q}= Q\circ D\circ Q|_{\IM Q}$ is invertible. Hence, it  is Fredholm and its index is equal to 0. We conclude that  
$Q\circ \B_{c\f,u}\circ Q|_{\IM Q}$ ($0\le u\le 1$) is a continuous family of Fredholm operators with 
\[
	\ind Q\circ \B_{c\f,u}\circ Q|_{\IM Q}\ = \ 0.
\]
From \eqref{E:indB=diag} we now obtain 
\begin{multline}\notag
	\ind \B_{c\f}\ = \ \ind \B_{c\f,1} \ = \ 
	\ind P\circ \B_{c\f,1}\circ P|_{\IM P}
	\\ = \ 
	\ind 
	  P^+\circ\big(i\frac{\p}{\p t}-ic\M_\f\big)\circ P^+\big|_{\IM P^+}
	\ + \ 
	\ind 
	  P^-\circ\big(-i\frac{\p}{\p t}-ic\M_\f\big)\circ P^-\big|_{\IM P^-}
	\\ = \ 
	\ind \big(\frac{\p}{\p t}-T_{f_t}^+\big)\big|_{L^2(S^1,\HH^+)}
	\ + \ 
	\ind \big(\frac{\p}{\p t}+T_{f_t}^-\big)\big|_{L^2(S^1,\HH^-)}.
\end{multline}
\end{proof}

\subsection{Proof of Theorem~\ref{T:sfToeplitz=Callias}}\label{SS:prsfToeplitz=Callias}
Theorem~\ref{T:sfToeplitz=Callias} follows now from \eqref{E:sf=ind}, \eqref{E:sf=-ind}, and \eqref{E:indddt=indTf}.\hfill$\square$


\addresseshere

\newcommand{\KK}{K\!K}
\newcommand{\mT}{\mathcal{T}}
\newcommand{\sA}{\mathscr{A}}
\newcommand{\Ran}{\operatorname{Ran}}

\newcommand{\mattwo}[4]{
  \left(\!\!\!\begin{array}{c@{~}c}#1&#2\\ #3&#4\\\end{array}\!\!\!\right)
}

\vspace{\baselineskip}\vspace{\baselineskip}
\appendix

{{\LARGE{\section*{\textbf{Appendix:\\ A perspective from (unbounded) KK-theory}}}}
\begin{center}
{by {\Large{Koen van den Dungen}}\\ Mathematisches Institut der Universit\"at Bonn, Endenicher Allee 60
D-53115 Bonn}\\ {\em E-mail address}: \texttt{kdungen@uni-bonn.de}
\end{center}
\renewcommand{\thesection}{A}

\setcounter{subsection}{0}
\setcounter{equation}{0}

\vspace{\baselineskip}
We consider the assumptions and notation of Section \ref{S:main}. The aim of this short appendix is to review Theorem \ref{T:sfToeplitz=Callias} from the perspective of (unbounded) $\KK$-theory \cite{Kas80,BJ83}. For simplicity, we will assume that $\f = \{f_t\}_{t\in S^1}$, viewed as an $M_k(\CC)$-valued function on $S^1\times M$, is chosen such that $(1+\f^2)^{-1}$ vanishes at infinity. This assumption ensures that the operator $\MM_\f$ (multiplication by $\f$), acting on the Hilbert $C_0(S^1\times M)$-module $\Gamma_0(S^1\times M,E\otimes\CC^k)$, has compact resolvents, so that $(\CC,\Gamma_0(S^1\times M,E^+\oplus E^-),\MM_\f)$ is an unbounded Kasparov $\CC$-$C_0(S^1\times M)$-module. It also means we do not need the (sufficiently large) constant $c>0$, and we simply set $c=1$. 

Theorem \ref{T:sfToeplitz=Callias} states that we have the equality
\begin{align}
\label{eq:sf_Toeplitz}
\spf(T_{f_t}^+) - \spf(T_{f_t}^-) = \ind \B_{\f} \in \ZZ .
\end{align}

In the context of $\KK$-theory, the right-hand-side of this equality should be viewed as an element in $\KK^0(\CC,\CC)$. 
The left-hand-side naturally defines an element in $\KK^1(\CC,C(S^1))$ (\emph{cf}.\ \cite[\S2.3]{Dun17}), given as the (odd!) class of the regular self-adjoint Fredholm operator 
\[
\mT_{\f} := \mattwo{T_{\f}^+}{0}{0}{-T_{\f}^-}
\]
on the Hilbert $C(S^1)$-module $C(S^1,(\HH^+\oplus\HH^-)\otimes\CC^k)$, where $T_\f^\pm = \{T_\f^\pm(t)\}_{t\in S^1}$ is given by $T_\f^\pm(t) := T_{f_t}^\pm$, and $\HH = \HH^+\oplus\HH^-$ denotes the kernel of $D$. 
Of course, these $\KK$-groups are both isomorphic to $\ZZ$, and we have a natural isomorphism $\cdot \otimes_{C(S^1)} [-i\partial_t] \colon \KK^1(\CC,C(S^1)) \to \KK^0(\CC,\CC)$ (which sends the spectral flow of a family $A(t)$ to the index of $\partial_t-A$, as described in \S\ref{SS:sf=index}). Thus we rewrite Eq.\ \eqref{eq:sf_Toeplitz} as (\emph{cf.}\ Eq.\ \eqref{L:indddt=indBf})
\[
\mT_{\f} \otimes_{C(S^1)} [-i\partial_t] = \ind \B_{\f} \in \KK^0(\CC,\CC) .
\]
Now let us consider the right-hand-side of this equality. It is well understood that the index class of the Callias-type operator $\B_{\f} = \D + i \MM_\f$ is given by the Kasparov product $[\B_{\f}] = [\MM_\f] \otimes_{C_0(S^1\times M)} [\D]$, cf.\ \cite{Bun95}. 
The class of $\D$ is simply given as the exterior Kasparov product $[\D] = [ D] \otimes [-i\partial_t]$ of the Dirac operator $ D$ on $M$ with $-i\partial_t$ on $S^1$. Using the properties of the Kasparov product, we then obtain
\[
\ind \B_{\f} = [\B_{\f}] = [\MM_\f] \otimes_{C_0(S^1\times M)} \Big( [ D] \otimes [-i\partial_t] \Big) = \Big( [\MM_\f] \otimes_{C_0(M)} [ D] \Big) \otimes_{C(S^1)} [-i\partial_t] .
\]
Since the Kasparov product with $[-i\partial_t]$ gives an isomorphism, Eq.\ \eqref{eq:sf_Toeplitz} can be rewritten as
\[
[\mT_{\f}] = [\MM_\f] \otimes_{C_0(M)} [ D] \in \KK^1(\CC,C(S^1)) .
\]
The Kasparov product on the right-hand-side can be computed \cite[Example 2.38]{BMS16}, and is represented by the regular self-adjoint operator (with compact resolvents)
\[
\sA_{\f} := \mattwo{\MM_{\f}^+}{ D^-}{ D^+}{-\MM_{\f}^-}
\]
on the Hilbert $C(S^1)$-module $C(S^1,L^2(M,E^+\oplus E^-))$. 
Theorem \ref{T:sfToeplitz=Callias} can then be reproven by showing the equality $[\mT_\f] = [\sA_\f]$ in $\KK^1(\CC,C(S^1))$.

\begin{proposition}
\label{prop:Braverman_odd}
We have the equality 
\[
[\mT_{\f}] = [\sA_{\f}] \in \KK^1(\CC,C(S^1)) .
\]
\end{proposition}
\begin{proof}
The proof is closely analogous to the proof of Lemma \ref{L:indddt=indBf}. 
Let $P=P^+\oplus P^-$ denote the projection onto the kernel of $ D$, and write $Q=1-P$. 
Since $P D P=0$, we have the equality $\mT_{\f} = P \sA_{\f} P$ (where we used the definition of the Toeplitz operators $T_{f_t} := P M_{f_t} P$). Hence we need to show that $P \sA_{\f} P$ and $\sA_{\f}$ define the same class in $\KK^1(\CC,C(S^1))$. 
By Corollary \ref{C:PMQcompact} we know that 
\[
P \sA_{\f} Q = \mattwo{P^+\MM_{\f}^+Q^+}{0}{0}{-P^-\MM_{\f}^-Q^-} 
\]
is compact, and similarly for $Q \sA_{\f} P$. 
This implies that $P \sA_{\f} P$ and $Q \sA_{\f} Q$ are both Fredholm, and that $[\sA_{\f}] = [P \sA_{\f} P] + [Q \sA_{\f} Q]$. 
Rescaling the function $\f$ by $c>0$, we see that the operator $Q \sA_{c\f} Q$ is Fredholm for any $c>0$. Furthermore, since $ D$ is invertible on $\Ran Q$, we find for $c=0$ that $Q \sA_{0\f} Q = Q D Q$ is invertible, and therefore its class in $\KK^1(\CC,C(S^1))$ is trivial. 
Since we have a continuous path of Fredholm operators for $0\leq c\leq1$, we conclude that the class of $Q \sA_{\f} Q$ is also trivial. Thus we obtain
\begin{align*}
[\sA_{\f}] &= [P \sA_{\f} P] + [Q \sA_{\f} Q] = [P \sA_{\f} P] .
\qedhere
\end{align*}
\end{proof}

The statement and proof of Proposition \ref{prop:Braverman_odd} do not rely on the notion of spectral flow, but merely consider the Fredholm operator $\mT_\f$ and its odd $\KK$-class. Hence Proposition \ref{prop:Braverman_odd} can straightforwardly be generalised to the case where we replace $S^1$ by an arbitrary compact space. 
We thus obtain the following:

\begin{theorem}
Let $E = E^+\oplus E^-$ be a graded Dirac bundle over a complete Riemannian manifold $M$, and let $ D$ be the associated Dirac operator. 
Assume that zero is an isolated point of the spectrum of $ D$, and let $P$ denote the projection onto the kernel of $ D$. 
Let $S$ be a compact topological space, and let $\f = \{f_t\}_{t\in S} \in C(S\times M,M_k(\CC))$ be given by a continuous family of smooth $M_k(\CC)$-valued functions $f_t$ on $M$ such that $(1+\f^2)^{-1}$ vanishes at infinity. We consider the Toeplitz operator $\mT_\f = (P\otimes1) \MM_\f (P\otimes1)$ on the Hilbert $C(S)$-module $C(S,\HH\otimes\CC^k)$. Then we have the equality
\[
[\mT_\f] = [\MM_\f] \otimes_{C_0(M)} [ D] \in \KK^1(\CC,C(S)) .
\]
\end{theorem}

\subsection*{Acknowledgements}
I would like to thank Maxim Braverman, Matthias Lesch, and Bram Mesland for an interesting discussion.


\end{document}